\documentclass[12pt]{amsart}
\usepackage{geometry}
\usepackage{graphicx}
\usepackage{amsmath, amsthm, amsfonts, amssymb, mathtools, datetime, graphicx, url, longtable}
\usepackage{tikz-cd}

\usepackage{array, tabularx, float, multirow}

\usepackage[utf8]{inputenc}
\usepackage[english]{babel}

\usepackage{multicol}
\usepackage{vwcol}

\usepackage{hyperref}
\hypersetup{
 colorlinks,
 linkcolor=blue,
 citecolor={blue}
 }
\usepackage[noabbrev,capitalize]{cleveref}

\newcommand{\CC}{\mathbb{C}}

\newcommand{\PP}{\mathbb{P}}

\newcommand{\ZZ}{\mathbb{Z}} 
\newcommand{\cE}{\mathcal{E}} 
\newcommand{\cN}{\mathcal{N}} 
\newcommand{\cO}{\mathcal{O}} 
\renewcommand{\P}{{\bf{P}}}

\newcommand{\on}[1]{\operatorname{#1}}

\newcommand{\Pic}{\on{Pic}}

\newcommand{\Proj}{\on{Proj}}
\newcommand{\Sym}{\on{Sym}}
\newcommand{\ch}{\ensuremath{\operatorname{ch}}}
\newcommand{\into}{\hookrightarrow}
\newcommand{\onto}{\twoheadrightarrow}

\newcommand{\Aut}{\on{Aut}}

\newtheorem{thm}{Theorem}[section]
\newtheorem{lemma}[thm]{Lemma}
\newtheorem{claim}[thm]{Claim}

\theoremstyle{definition}
\newtheorem{rmk}[thm]{Remark}

\begin{document}

\title{Horospherical 2-Fano varieties}

\author[Carolina Araujo]{Carolina Araujo}
\address{\sc Carolina Araujo\\
IMPA\\
Estrada Dona Castorina 110\\
22460-320 Rio de Janeiro\\ Brazil}
\email{caraujo@impa.br}

\author[Ana-Maria Castravet]{Ana-Maria Castravet}
\address{Ana-Maria Castravet, 
Universit\'e Paris-Saclay, UVSQ, Laboratoire de Math\'ematiques de Versailles, 
45 Avenue des \'Etats Unis, 78035 Versailles, France }
\email{ana-maria.castravet@uvsq.fr}

\date{}

\begin{abstract}
We classify $2$-Fano horospherical varieties with Picard number $1$.
We also review all the known examples of $2$-Fano manifolds 
and investigate the relation between the $2$-Fano condition and 
different notions of stability.
\end{abstract}

\maketitle

\section{Introduction}\label{introduction}

\emph{Fano manifolds} are complex projective manifolds $X$ whose tangent bundles have ample first Chern class $-K_X=c_1(T_X)$.
They play a distinguished role in the birational classification of complex projective varieties, and 
satisfy formidable properties.
For instance, Fano manifolds are rationally connected (\cite{KMM}), i.e., any two points can be connected by a rational curve.
They also satisfy special arithmetic properties:  algebraic families of Fano manifolds over $1$-dimensional bases always have sections (\cite{GHS}). 

In \cite{dJ-S:2fanos_1}, De Jong and Starr introduced a special subclass of Fano manifolds: {\em{$2$-Fano manifolds}} are 
Fano manifolds $X$ with positive second Chern character,
$$\ch_2(T_X)=\frac{1}{2}c_1(T_X)^2-c_2(T_X)>0.$$
This means that $\ch_2(T_X)\cdot S>0$ for every projective surface $S\subset X$.
The $2$-Fano manifolds are expected to enjoy stronger versions of the nice properties of Fano manifolds. This expectation has been confirmed in several special cases. 
For instance, under some conditions, 
$2$-Fano manifolds are covered by rational surfaces (\cite{dJ-S:2fanos_2}, \cite{2fanos}), and 
it is expected that families of $2$-Fano manifolds over $2$-dimensional bases admit meromorphic sections (modulo Brauer obstruction).

On the other hand, very few examples of $2$-Fano manifolds are known:
complete intersections in weighted projective spaces, certain rational homogeneous spaces with Picard number $1$
and some of their linear sections, and a few two-orbit varieties. 
We refer to Section~\ref{section_K-stability} for a survey of all  the known examples of $2$-Fano manifolds.

This paper was motivated by the goal of classifying  $2$-Fano manifolds. 
A natural class of varieties where to look for new examples of $2$-Fano manifolds 
is that of {\it spherical varieties}. These are projective manifolds $X$ admitting an action of  a reductive group
$G$, with a Borel subgroup $B$, such that $X$ has a dense $B$-orbit. 
The class of spherical varieties includes in particular rational homogeneous spaces and toric varieties. 
The classification of rational homogeneous $2$-Fano manifolds with Picard number $1$ was obtained in  \cite{hfanos}.
There has been progress in the classification of toric $2$-Fano manifolds 
(\cite{Nobili2011},  \cite{Sato2012}, \cite{Sato2016}, \cite{SatoSuyama2020}, \cite{SanoSatoSuyama2020},  \cite{toric_2fanos}).
The only known examples of toric $2$-Fano manifolds are projective spaces.

Another important subclass of spherical varieties is that of {\it horospherical varieties}.
These are special spherical varieties with nice geometric and representation theoretic properties. 
They can be defined as spherical $G$-varieties for which the stabiliser of a point in the dense orbit contains
a conjugate of the maximal unipotent subgroup of a Borel subgroup.
In this paper, we investigate the $2$-Fano condition for horospherical varieties  with Picard number $1$.
These varieties were classified by Pasquier in \cite{Pasquier}.
They are either rational homogeneous, or their isomorphism classes are uniquely determined by one
the following triples $(\text{Type}(G), \omega_Y , \omega_Z)$, where  Type$(G)$ is the semisimple Lie type of the reductive group $G$, 
and $\omega_Y$ and $\omega_Z$ are fundamental weights (see Section~\ref{section_horospherical} for details):
\begin{enumerate}
\item[($1$)] $(B_n, \omega_{n-1} , \omega_n)$, with $n\geq 3$;
\item[($2$)] $(B_3, \omega_{1} , \omega_3)$;
\item[($3$)] $(C_n, \omega_{m} , \omega_{m-1})$, with $n\geq 2$ and $2\leq m \leq n$; 
\item[($4$)] $(F_4, \omega_{2} , \omega_3)$; 
\item[($5$)] $(G_2, \omega_{1} , \omega_2)$.
\end{enumerate}

\noindent We compute the second Chern character of each variety in Pasquier's list and verify whether the 
 $2$-Fano condition holds for these manifolds.
As a result, we obtain the following classification of $2$-Fano horospherical varieties with Picard rank $1$. 

\begin{thm}\label{main_thm}
The only non-homogeneous $2$-Fano horospherical varieties with Picard rank $1$ are the one of type (2),  
and  the ones of type (3) with $(n,m)=(3k,2k+1)$ for some $k\geq 1$.
\end{thm}

The horospherical variety of type (2) can also be described as a general hyperplane section of 
the $10$-dimensional orthogonal Grassmannian $OG_+(5,10)$ (\cite[Proposition 1.16]{Perrin+}).
It was proved to be $2$-Fano in \cite[Proposition 34]{2fanos2}.
The horospherical varieties of type (3) are known as {\it odd symplectic Grassmannians}, and
those with $(n,m)=(3k,2k+1)$ for some $k\geq 1$ were already known to be $2$-Fano from \cite[Example 5.6]{2fanos}. 

This paper is organized as follows. 
In Section~\ref{section_horospherical}, we provide a background on horospherical varieties.
In Section~\ref{section_proof}, we check the $2$-Fano condition for each variety in Pasquier's list.
In Section~\ref{section_K-stability}, we investigate the relation between the $2$-Fano condition, $K$-stability
and (slope) stability of the tangent bundle.

\medskip

\noindent {\bf Conventions and notation.} 
Throughout this paper we work over the field $\CC$ of complex numbers.
Given a vector bundle $\cE$ on a variety $Z$, we follow Grothendieck's notation and
denote by  $\PP(\cE)$ the projective bundle over $Z$ of one-dimensional quotients 
of the fibers of $\cE$, i.e., $\PP(\cE)=\Proj(\Sym \cE)$. 
On the other hand, given a complex vector space $V$, we denote by $\P(V)$
the projective space of one-dimensional linear subspaces of $V$, so that  $\P(V)= \PP(V^*)$.

\

\paragraph{\textbf{Acknowledgements}}
Carolina Araujo was partially supported by grants from CNPq, Faperj and CAPES/COFECUB. 
Ana-Maria Castravet was supported by the ANR grant FanoHK and the Laboratoire de Math\'ematiques de Versailles. 
Most of this work was developed during Carolina Araujo's visit to the Laboratoire de Math\'emati\-ques de Versailles, 
funded by the ``Brazilian-French Network in Mathematics''.
We are grateful to this program for the financial support. We thank Nicolas Perrin for many useful explanations about horospherical varieties, and our collaborators Roya Beheshti, Kelly Jabbusch, Svetlana Makarova, Enrica Mazzon and Nivedita Viswanathan 
for many rich discussions about $2$-Fano manifolds. 
This paper was conceived as a contribution to ``Edge Volume: 2018-2022'' after our participation in Edge Days 2022.
We thank the organizers and participants of the several Edge Days for making it such a nice conference series. 
Special thanks to Vanya Cheltsov for his aggregating energy.

\section{Horospherical varieties}\label{section_horospherical}

Let $G$ be a complex connected reductive group, and $X$ a spherical
$G$–variety, i.e.,  $X$ has a dense $B$-orbit for a Borel subgroup $B\subset G$.
We say that $X$ is a {\it horospherical variety} if the stabiliser of a point in the dense $B$-orbit contains
a conjugate of the maximal unipotent subgroup of $B$.

Horospherical varieties with Picard number $1$  were classified by Pasquier in \cite{Pasquier}.
They are either rational homogeneous, or their isomorphism classes are uniquely determined by one of
the following triples $(\text{Type}(G), \omega_Y , \omega_Z)$, 
where  Type$(G)$ is the semisimple Lie type of the reductive group $G$, 
and $\omega_Y$ and $\omega_Z$ are fundamental weights:
\begin{enumerate}
\item[($1$)] $(B_n, \omega_{n-1} , \omega_n)$, with $n\geq 3$;
\item[($2$)] $(B_3, \omega_{1} , \omega_3)$;
\item[($3$)] $(C_n, \omega_{m} , \omega_{m-1})$, with $n\geq 2$ and $2\leq m \leq n$; 
\item[($4$)] $(F_4, \omega_{2} , \omega_3)$; 
\item[($5$)] $(G_2, \omega_{1} , \omega_2)$.
\end{enumerate} 

\noindent Following \cite{Perrin+}, we denote by $X^1(n)$, $X^2$, $X^3(n,m)$, $X^4$ and $X^5$ the corresponding horospherical varieties, where the superscript indicates the class that a variety belongs to. 
We refer to this list of varieties as Pasquier’s list.
Next we explain the geometric characterization of non–homogeneous horospherical varieties from \cite{Pasquier}.  
From now on, $X$ always denotes a smooth projective non–homogeneous horospherical variety of Picard number $1$ 
associated to the triple $(\text{Type}(G), \omega_Y , \omega_Z)$. 
Let $V_Y$ and $V_Z$ be the irreducible $G$–representations with highest weights $\omega_Y$ and $\omega_Z$,
respectively, and $v_Y$ and $v_Z$ the corresponding lowest weight vectors.
We denote by $P_Y\subset G$ the stabilizer of $[v_Y]$ in $\P(V_Y)$, and by $P_Z\subset G$ the stabilizer of $[v_Z]$ in $\P(V_Z)$. 
They are maximal parabolic subgroups of $G$, so that the quotients $G/P_Y$ and $G/P_Z$ are rational homogeneus spaces 
of Picard number $1$.
We assume that both $P_Y$ and $P_Z$ contain the same Borel subgroup $B$.
The horospherical variety $X$ is the closure in $\P(V_Y\oplus V_Z)$ of the $G$–orbit of $[v_Y + v_Z]$.
The group $G$ acts on $X$ with two closed orbits, $Y\cong G/P_Y$ and $Z\cong G/P_Z$, and one dense open orbit 
$U=X\setminus (Y\cup Z)$.
We denote by $\cN_Y$ the normal bundle of $Y$ in $X$, and by $\cN_Z$ the normal bundle of $Z$ in $X$.
Let $\pi_Y:\tilde X_Y \to X$ be the blow-up of $X$ along $Y$, and denote by $E_Y\subset \tilde X_Y$ the exceptional divisor.
Then there is a natural projective space bundle $p_Z:\tilde X_Y \to Z$. 
As a projective bundle over $Z$, we have an isomorphism $\tilde X_Y\cong\PP(\cE_Y)$, where $\cE_Y={\cN_Y}^*\oplus\cO_Y$.
The exceptional divisor $E_Y$ is isomorphic to $G/(P_Y \cap P_Z)$, and the projections to $Y$ and $Z$ (via 
the restrictions of $\pi_Y$ and $p_Z$, respectively) correspond to the natural projections $G/(P_Y \cap P_Z)\to G/P_Y$ and $G/(P_Y \cap P_Z)\to G/P_Z$.
The situation is illustrated in the following diagram: 
 \[
  \begin{tikzpicture}[xscale=1.5,yscale=-1.2]
    \node (A0_0) at (0, 0) {$E_Y$};
    \node (A1_0) at (1, 0) {$\tilde X_Y$};
    \node (A2_0) at (2, 0) {$Z$};
    \node (A0_1) at (0, 1) {$Y$};
    \node (A1_1) at (1, 1) {$X$};
    \path (A0_0) edge [->]node [auto] {} (A1_0);
    \path (A0_0) edge [->]node [auto] {} (A0_1);
    \path (A0_0) edge [->,bend right=45]node [auto] {}(A2_0);
    \path (A1_0) edge [->]node [auto] {$\scriptstyle{\pi_Y}$} (A1_1);
    \path (A1_0) edge [->]node [auto] {$\scriptstyle{p_Z}$} (A2_0);
    \path (A0_1) edge [->]node [auto] {} (A1_1);
  \end{tikzpicture}
\]

\noindent Similarly, denoting by $\pi_Z:\tilde X_Z \to X$ the blow-up of $X$ along $Z$, with exceptional divisor 
$E_Z\subset \tilde X_Z$, there is a natural projective space bundle $p_Y:\tilde X_Z \to Y$.
As a projective bundle over $Y$, we have an isomorphism $\tilde X_Z\cong\PP(\cE_Z)$, where $\cE_Z={\cN_Z}^*\oplus\cO_Z$.
The exceptional divisor $E_Z$ is isomorphic to $G/(P_Y \cap P_Z)$, and the projections to $Y$ and $Z$ (via 
the restrictions of $p_Y$ and $\pi_Z$, respectively) correspond to the natural projections $G/(P_Y \cap P_Z)\to G/P_Y$ and $G/(P_Y \cap P_Z)\to G/P_Z$.
The situation is illustrated in the following diagram: 
 \[
  \begin{tikzpicture}[xscale=1.5,yscale=-1.2]
    \node (A0_0) at (0, 0) {$E_Z$};
    \node (A1_0) at (1, 0) {$\tilde X_Z$};
    \node (A2_0) at (2, 0) {$Y$};
    \node (A0_1) at (0, 1) {$Z$};
    \node (A1_1) at (1, 1) {$X$};
    \path (A0_0) edge [->]node [auto] {} (A1_0);
    \path (A0_0) edge [->]node [auto] {} (A0_1);
    \path (A0_0) edge [->,bend right=45]node [auto] {}(A2_0);
    \path (A1_0) edge [->]node [auto] {$\scriptstyle{\pi_Z}$} (A1_1);
    \path (A1_0) edge [->]node [auto] {$\scriptstyle{p_Y}$} (A2_0);
    \path (A0_1) edge [->]node [auto] {} (A1_1);
  \end{tikzpicture}
\]

\begin{rmk}\label{Aut(X)}
In the above geometric description of the non–homogeneous horospherical variety $X$, the roles of $Y$ and $Z$ are interchangeable.
However, there is an important geometric distinction between them. 
The automorphism group $\Aut(X)$ is a semi–direct product of $G$ with its unipotent radical, and it acts on $X$ with two orbits: 
the closed orbit $Z$ and the dense open orbit $X\setminus Z = U\cup Y$.
These varieties are often referred to as \emph{two-orbit varieties}.
\end{rmk}

We end this section by collecting in Table~\ref{table} below some numerical invariants of $X$, $Y$ and $Z$ 
for each horospherical variety in Pasquier’s list. 
In what follows, we denote by $c_1(X)$ the index of the Fano variety $X$, so that 
$-K_X = c_1(X)\cdot H_X$,
where $H_X$ is the ample generator of $\Pic(X)$, 
and similarly for $Y$ and $Z$. We denote by $c_{Y/X}$ and $c_{Z/X}$ the codimensions of $Y$ and $Z$ in $X$, respectively. 

\begin{table}[h] \label{table}
\begin{tabular}{|c||c|c|c|c|c|}
\hline
Type       & \textbf{$c_1(X)$} & \textbf{$c_1(Y)$} & \textbf{$c_1(Z)$} & \textbf{$c_{Y/X}$} & \textbf{$c_{Z/X}$} \\ \hline \hline
$X^1(n)$   & $n+2$             & $n+1$             & $2n$              & $2$                & $n$                \\ \hline
$X^2$      & $7$               & $5$               & $6$               & $4$                & $3$                \\ \hline
$X^3(n,m)$ & $2n+2-m$          & $2n+1-m$          & $2n+2-m$          & $m$                & $2(n+1-m)$         \\ \hline
$X^4$      & $6$               & $5$               & $7$               & $3$                & $3$                \\ \hline
$X^5$      & $4$               & $3$               & $5$               & $2$                & $2$                \\ \hline
\end{tabular}
\medskip
\caption{Numerical invariants of $X$, $Y$ and $Z$.}
\end{table}

\section{The $2$-Fano condition}\label{section_proof}

In this section, we check the $2$-Fano condition for each non-homogeneous horospherical
variety $X$ in Pasquier's list.
As we saw in Section~\ref{section_horospherical}, for each $X$ there is a complex connected reductive group $G$
acting on $X$ with two closed orbits, $Y$ and $Z$, and one dense open orbit $U=X\setminus (Y\cup Z)$.
Moreover, the blow-up of $X$ along $Y$ admits a structure of projective space bundle over $Z$: 
 \[
  \begin{tikzpicture}[xscale=1.5,yscale=-1.2]
    \node (A1_0) at (1, 0) {$\tilde X_Y$};
    \node (A2_0) at (2, 0) {$Z$.};
    \node (A1_1) at (1, 1) {$X$};
    \path (A1_0) edge [->]node [auto] {$\scriptstyle{\pi_Y}$} (A1_1);
    \path (A1_0) edge [->]node [auto] {$\scriptstyle{p_Z}$} (A2_0);
  \end{tikzpicture}
\]
(Similarly, the blow-up of $X$ along $Z$ admits a structure of projective space bundle over $Y$.)
In order to compute the second Chern character of $X$, we first relate it to the second Chern character of 
the blow-up $\tilde X_Y$ using Lemma~\ref{Lem:ch_blowup} below. 
We then compute the second Chern character of $\tilde X_Y$
using the projective space bundle structure $p_Z:\tilde X_Y \to Z$ and Lemma~\ref{P-bundle} below. 

The following formula relates the second Chern character of a variety with its blow-up. 

\begin{lemma}\cite[Lemma 5.1]{dJ-S:2fanos_1}
\label{Lem:ch_blowup}
Let $X$ be a smooth projective variety, and $Y\subset X$ a smooth projective subvariety of codimension 
$c \geq 2$ and normal bundle $\mathcal{N}_{Y/X}$. 
Let $\pi_Y:\tilde X_Y \to X$ be the blow-up of $X$ along $Y$, with  exceptional divisor $E\subset \tilde X_Y$, 
and set $\sigma={\pi_Y}_{|E}$.
 \[
  \begin{tikzpicture}[xscale=1.5,yscale=-1.2]
    \node (A0_0) at (0, 0) {$E$};
    \node (A1_0) at (1, 0) {$\tilde X_Y$};
    \node (A0_1) at (0, 1) {$Y$};
    \node (A1_1) at (1, 1) {$X$};
    \path (A0_0) edge [->]node [auto] {$\scriptstyle{j}$} (A1_0);
    \path (A0_0) edge [->]node [auto] {$\scriptstyle{\sigma}$} (A0_1);
    \path (A1_0) edge [->]node [auto] {$\scriptstyle{\pi_Y}$} (A1_1);
    \path (A0_1) edge [->]node [auto] {} (A1_1);
  \end{tikzpicture}
\]
Then
$$ ch_2 (\tilde X_Y) =
\pi_Y^* ch_2 (X) + \frac{c+1}{2} E^2 - j_*\sigma^* c_1(\mathcal{N}_{Y/X})
.$$
\end{lemma}

The following formula computes the second Chern character of a projective space bundle. 
 
\begin{lemma}\label{B0}\cite[Lemma 4.1]{dJ-S:2fanos_1}\label{P-bundle}
Let $Z$ be a smooth projective variety and let $\cE$ be a rank $r$ vector bundle on $Z$. Denote by $p_Z:\PP(\cE)\to Z$ the natural projection and set $\xi=c_1(\cO_{\PP(\cE)}(1))$. 
Then
$$ch_2(\PP(\cE)) =p_Z^*\big(ch_2(Z)+ch_2(\cE)\big)-p_Z^*c_1(\cE)\cdot\xi+\frac{r}{2}\xi^2.$$
\end{lemma}

\noindent (We note that in \cite{dJ-S:2fanos_1} the notation $\PP\cE$ stands for $\Proj(\Sym \cE^*)$, and 
this accounts for the sign difference between the formula in Lemma~\ref{P-bundle} and the one from 
\cite[Lemma 4.1]{dJ-S:2fanos_1}.)

Applying  Lemma \ref{Lem:ch_blowup} to the  blow-up $\pi_Y:\tilde X_Y\to X$ yields:
\begin{equation} \label{formula_blowup}
ch_2(\tilde X_Y)=\pi_Y^* ch_2(X)+\frac{c_{Y/X}+1}{2}E^2_Y-{j_Y}_*\pi_Y^*c_1(\cN_Y),    
\end{equation}
where $j_Y:E_Y\into \tilde X_Y$ denotes the natural inclusion. 

Recall that, as a projective space bundle over $Z$, we have $\tilde X_Y\cong\PP(\cE_Z)$, where
$\cE_Z\cong{\cN_Z}^*\oplus\cO_Z$  and $\cN_Z$ is the normal bundle of $Z$ in $X$. 
We write $\xi_Z:=c_1(\cO_{\PP(\cE_Z)}(1))$. 
Applying Lemma \ref{P-bundle} to the projective bundle $p_Z: \PP(\cE_Z)\to Z$  yields:
\begin{equation} \label{formula_PE}
ch_2(\tilde X_Y)=p_Z^*\big(\ch_2(Z)+\ch_2(\cE_Z)\big)-p_Z^*c_1(\cE_Z)\cdot\xi+\frac{c_{Z/X} +1}{2}\xi_Z^2.
\end{equation}

Let $\tilde{\PP}\subseteq \P(V_Y\oplus V_Z)\times\P(V_Z)$ be the blow-up of $\P(V_Y\oplus V_Z)$ along $\P(V_Y)$.
As projective space bundle over $\P(V_Z)$, we have $\tilde{\PP}\cong \PP_{\P(V_Z)}((\cO(-1)\otimes V_Y)\oplus\cO)$. 
Let $\tilde{F}_z\cong \P^{\dim(V_Y)}$ be the fiber of $\tilde{\PP}\to\P(V_Z)$ over a general point $z\in Z\subset \P(V_Z)$. 
The blow-up map $\tilde{\PP}\to \P(V_Y\oplus V_Z)$ is an isomorphism along $\tilde{F}_z$. 
The image of $\tilde{F}_z$ in $\P(V_Y\oplus V_Z)$ is the linear subspace spanned by $\P(V_Y)$ and the point $z\in Z\subset\P(V_Z)$. 
Note that  $\tilde X_Y$ is the proper transform  of $X$ in $\tilde{\PP}$, and the map $p_Z: \tilde X_Y\to Z$ is the restriction to $\tilde X_Y$ of the second projection $\tilde{\PP}\to\P(V_Z)$.
Consider the fiber $F_z\cong \PP^{c_{Z/X}}$ of the projective bundle $p_Z: \tilde X_Y\to Z$ over the point $z\in Z$. 

\begin{claim}\label{linear subspace}
The inclusion $F_z\subseteq \tilde{F}_z$ embeds $F_z\cong \PP^{c_{Z/X}}$ as a linear subspace in $\tilde{F}_z\cong\PP^{\dim(V_Y)}$.
\end{claim}

\begin{proof}
The preimage of $Z$ via the second projection $\tilde{\PP}\to\P(V_Z)$ is the projective bundle $\PP_{Z}((\cO_Z(-1)\otimes V_Y)\oplus\cO_Z)\to Z$. 
The inclusion
$\PP(\cE_Z)\cong\tilde X_Y\subset\tilde{\PP}$ comes from the canonical embedding 
$${\cE_Z}^*\cong\cN_Z\oplus\cO_Z\subset {\cN_{\P(V_Z)/\P(V_Y\oplus V_Z)}}_{|Z}\oplus\cO_Z\cong (\cO_Z(1)\otimes V_Y)\oplus\cO_Z,$$
and the claim follows. 
\end{proof}

In all cases, $c_{Z/X}\geq 2$. Let $S'_Z$ be a general linear subspace of dimension $2$ in $F_z\cong\PP^{c_{Z/X}}$ and let $S_Z$ be its image in $X$. We compute $\ch_2(\tilde X_Y)\cdot S'_Z$ and $\ch_2(X)\cdot S_Z$ from the above formulae. 
Clearly, $S'_Z\cdot \xi_Z^2=1$ and, as ${p_Z}_*S'=0$, formula \eqref{formula_PE} gives
$$ch_2(\tilde X_Y)\cdot S'_Z=\frac{c_{Z/X} +1}{2}.$$

The inclusion $E_Y=\PP_Z(\cN_Z^*)\subset \PP_Z(\cE_Z)= \tilde X_Y$
comes from the natural surjection 
$$
\cE_Z\cong \cN_Z^*\oplus\cO_Z \onto \cN_Z^*.
$$
It follows that the scheme theoretic intersection $\ell_Z:=S'_Z\cap E_Y$ is a line in $S'_Z\cong\PP^2$.  
Note that by Claim \ref{linear subspace}, the image of $\ell_Z$ in $X\subseteq \P(V_Y\oplus V_Z)$ is a line.  
It follows that $E_Y^2\cdot S'_Z=1$. Next we compute
$$S'_Z\cdot \big({j_Y}_*\pi_Y^*c_1(\cN_Y)\big)=\big(\ell_Z\cdot j_Y^*\pi_Y^*c_1(\cN_Y)\big)_{E_Y}=({\pi_Y}_*\ell_Z)\cdot c_1(\cN_Y).$$
Since $c_1(\cN_Y)=c_1(T_X)_{|Y}-c_1(T_Y) = (c_1(X)-c_1(Y))\cdot H_{|Y}$, and $H_{|Y}\cdot\ell_Z=1$, we get  
$$S'_Z\cdot \big({j_Y}_*\pi_Y^*c_1(\cN_Y)\big)=c_1(X)-c_1(Y).$$
It follows from (\ref{formula_blowup}) and (\ref{formula_PE}) that 
$$ch_2(\tilde X_Y)\cdot S'_Z=ch_2(X)\cdot S_Z+\frac{c_{Y/X}+1}{2}-\big(c_1(X)-c_1(Y)\big),$$
and we obtain the following theorem: 
\begin{thm}
Let $X$ be one of the non-homogeneous horospherical varieties in Pasquier's list. In the notations of this section and Section \ref{section_horospherical}, 
if $S_Z$ is the image in $X$ of a general plane in a general fiber of the projective bundle 
$p_Z: \tilde X_Y\to Z$, then 
$$ch_2(X)\cdot S_Z=\frac{c_{Z/X}-c_{Y/X}}{2}+\big(c_1(X)-c_1(Y)\big).$$
Similarly, if $S_Y$ is the image in $X$ of a general plane in a general fiber of the projective bundle 
$p_Y: \tilde X_Z\to Y$, then we have
$$ch_2(X)\cdot S_Y=\frac{c_{Y/X}-c_{Z/X}}{2}+\big(c_1(X)-c_1(Z)\big).$$
and the following intersection numbers:
\begin{itemize}
    \item [(1) ] $ch_2(X)\cdot S_Y=-3(n-2)/2$, $ch_2(X)\cdot S_Z=n/2$, 
    \item [(2) ] $ch_2(X)\cdot S_Y=3/2$, $ch_2(X)\cdot S_Z=3/2$, 
    \item [(3) ] $ch_2(X)\cdot S_Y=(-2n-2+3m)/2$, $ch_2(X)\cdot S_Z=(2n+4-3m)/2$, 
    \item [(4) ] $ch_2(X)\cdot S_Y=-1$, $ch_2(X)\cdot S_Z=1$, 
    \item [(5) ] $ch_2(X)\cdot S_Y=-1$, $ch_2(X)\cdot S_Z=1$.
\end{itemize}
In particular, the only $2$-Fano non-homogeneous horospherical varieties in Pasquier's list are the varieties $X^2$ and $X^3(n,m)$ for $(n,m)=(3k,2k+1)$ for some $k\geq 1$.
\end{thm}

\begin{proof}
The formulae for $ch_2(X)\cdot S_Z$ and $ch_2(X)\cdot S_Y$ follow from the preceeding discussion. 
The intersection numbers then follow immediately from Table \ref{table}. Note that in Cases (1), (4) and (5) we have $ch_2(X)\cdot S_Y<0$, and hence $X$ is not $2$-Fano. 
In case (2), it follows from \cite[Fact 1.8]{Perrin+} and \cite[Lemma 3.1]{hfanos} that $b_4(X)=1$. 
Since $ch_2(X^2)\cdot S_Y=ch_2(X^2)\cdot S_Z=3/2$, we conclude that $X^2$ is $2$-Fano. 
This also follows from \cite[Proposition 34]{2fanos2}. 
In Case (3), the two conditions 
$ch_2(X)\cdot S_Y>0$ and  $ch_2(X)\cdot S_Z>0$ hold if and only if $3m=2n+3$. 
The latter condition is equivalent to $(n,m)=(3k,2k+1)$ for some $k\geq 1$.
Since the classes of the surfaces $S_Y$ and $S_Z$ generate $\text H_4(X,\ZZ)$ 
(see \cite[Section 1.7]{Perrin+}),  we conclude that $X^3(3k,2k+1)$ is $2$-Fano. 
This also follows from \cite[Example 5.6]{2fanos}. 
\end{proof}

\section{2-Fano varieties and stability}\label{section_K-stability}

We end this paper by reviewing all the known examples of $2$-Fano manifolds, 
and investigate the relation between the $2$-Fano condition and stability conditions. 
The notion of $K$-polystability for Fano manifolds has become very proeminent,  
specially since the establishment of the Yau-Tian-Donaldson conjecture, 
which asserts that the existence of a K\"ahler-Einstein metric on a Fano manifold is equivalent to its $K$-polystability. 
As we shall see, even though the $2$-Fano condition seems very restrictive, it does not imply $K$-polystability. 
Another related but weaker condition is the (slope) stability of the tangent bundle. 
Until very recently, there were no known examples of Fano manifolds with Picard rank $1$ and non-stable tangent bundle.
In \cite{Kanemitsu}, Kanemitsu verified which horospherical varieties have stable tangent bundle, presenting 
the first examples of Fano manifolds of Picard rank $1$ and non-stable tangent bundle.
We do not know of any example of $2$-Fano manifold with unstable tangent bundle, and one may ask whether 
the $2$-Fano condition implies stability of the tangent bundle.

All the known examples of $2$-Fano manifolds fall under one of three categories: 
complete intersections in weighted projective spaces, horospherical varieties with Picard number $1$, 
and linear sections of rational homogeneous spaces. We briefly discuss each of these classes. 
It is remarkable that all the known examples have Picard number $1$.

\subsection{Complete intersections in weighted projective spaces.}\label{ci_in_weighted} 
Let $\PP(a_0,\ldots, a_n)$ be a weighted projective space, and assume that 
$\gcd(a_0,\ldots, \hat a_i, \ldots a_n)=1$ for every $i\in\{0, \ldots, n\}$.
Denote by $H$ the effective generator of the class group of $\PP(a_0,\ldots, a_n)$. 
Let  $X$ be a smooth complete intersection of hypersurfaces with classes $d_1H$,..., $d_cH$
in $\PP(a_0,\ldots, a_n)$. 
Then the Chern character of $X$ is given by
$$\ch(X)=(n-c)+\sum_{k=1}^n\frac{a_0^k+\ldots+a_n^k-\sum d_i^k}{k!}c_1(H_{|X})^k.$$
It follows that $X$ is $2$-Fano if and only if $\sum d_i^2<\sum a_i^2$.

By \cite[Corollary 0.3]{PW}, the tangent bundle $T_X$ of any such complete intersection $X$ is stable 
as long as $\dim(X)\geq 3$.

\subsection{Horospherical varieties with Picard number $1$.}\label{Horospherical} 

We recall that horospherical varieties with Picard number $1$  were classified by Pasquier in \cite{Pasquier}.
They are either rational homogeneous, or are two-orbit varieties belonging to one of the 5 classes in Pasquier's list.
The classification of rational homogeneous $2$-Fano manifolds with Picard number $1$ was obtained in  \cite{hfanos}.
Together with Theorem~\ref{main_thm}, it yields the following classification of $2$-Fano
horospherical varieties with Picard number $1$. 
We refer to \cite[Section 3]{hfanos} for the notation regarding rational homogeneous spaces.

\begin{thm} \label{thm:horospherical}
The following is the complete list of $2$-Fano horospherical varieties with Picard number $1$:
\begin{enumerate}
\item Rational homogeneous spaces:
\begin{enumerate}
\item[-] $A_n/P^k$, for $k=1,n$ and for $n= 2k-1,2k$ when $2 \leq k \leq \frac{n+1}{2}$;
\item[-] $B_n/P^k$, for $k=1,n$ and for $2n=3k+1$ when $2 \leq k \leq n-1$;
\item[-] $C_n/P^k$, for $k=1,n$ and for $2n=3k-2$ when $2 \leq k \leq n-1$;
\item[-] $D_n/P^k$, for $k=1,n-1,n$ and for $2n=3k+2$ when $2 \leq k < n-1$;
\item[-] $E_n/P^k$, for $n=6,7,8$ and $k = 1,2,n$;
\item[-]$F_4/P^4$;
\item[-]$G_2/P^k$, for $k=1,2$.
\end{enumerate}
\item Two-orbit varieties from Pasquier's list:
\begin{enumerate}
\item[-] $X^2$;
\item[-] $X^3(3k,2k+1)$ for $k\geq 1$.
\end{enumerate}
\end{enumerate}
\end{thm}

\noindent We recall that $X^2$ is isomorphic to a general hyperplane section of the $10$-dimensional orthogonal Grassmannian $OG_+(5,10)$, while each $X^3(n,m)$ is an odd symplectic Grassmannian.

All rational homogeneous spaces are $K$-polystable and have stable tangent bundle. 
On the other hand, as already mentioned in \ref{Aut(X)}, the automorphism group $\Aut(X)$ of any two-orbit variety $X$ from Pasquier's list is non-reductive. This implies that $X$ is not $K$-polystable by Matsushima's criterion, which asserts that  
the existence of a K\"ahler-Einstein metric on a Fano manifold implies the reductivity of its automorphism group. 

Kanemitsu investigated the stability of the tangent bundle of the two-orbit varieties in Pasquier's list in \cite{Kanemitsu}.
The ones with stable tangent bundle are precisely the following:  $X^1(3)$, $X^2$, $X^3(n,m)$, with $n\geq 2$ and $2\leq m \leq n$, and $X^5$. 
In particular, we see that all $2$-Fano horospherical varieties with Picard number $1$ have stable tangent bundle.

\subsection{Linear sections of rational homogeneous spaces.}\label{linear_sections} 
Some linear sections of rational homogeneous spaces of Picard rank $1$ under their primitive embeddings are known to be 
$2$-Fano. Here is the known list: 
\begin{enumerate}
\item Let $X$ be a general codimension $c$ linear section of the Grassmannian $G(k,n)$ under the Pl\"ucker embedding,
with $2\leq k\leq\frac{n}{2}$. Then $X$ is $2$-Fano if and only if $n=2k$ and $c\leq 1$ by \cite[Proposition 32]{2fanos2}. 
\item Let $X$ be a general codimension $c$ linear section of the orthogonal Grassmannian $OG_+(k,2k)$ 
under the half-spinor embedding, with $k\geq 4$. Then $X$ is $2$-Fano if and only if $c\leq 3$ by \cite[Proposition 34]{2fanos2}.
\end{enumerate}


\begin{thebibliography}{xxxxxxx}{}


\bibitem[ABC+22]{hfanos}
{C. Araujo, R. Beheshti, A.-M. Castravet, K. Jabbusch, S. Makarova, E. Mazzon, L. Taylor, N. Viswanathan {\it Higher Fano Manifolds}, Rev. Un. Mat. Argentina {\bf 64} (2022), no. 1, 103--125.}

\bibitem[ABC+23]{toric_2fanos}
{C. Araujo, R. Beheshti, A.-M. Castravet, K. Jabbusch, S. Makarova, E. Mazzon, N. Viswanathan {\it The minimal projective bundle dimension and toric $2$-Fano manifolds}, arXiv:2301.00883 (2023).}

\bibitem[AC12]{2fanos}
{C. Araujo, A.-M. Castravet,  {\it Polarized minimal families of rational curves and higher {F}ano manifolds}, 
American Journal of Mathematics {\bf 134} 87--107, (2012).}


\bibitem[AC13]{2fanos2}
{C. Araujo, A.-M. Castravet,  {\it Classification of 2-{F}ano manifolds with high index}. A celebration of algebraic geometry, 1--36, Clay Math. Proc., 18, 
Amer. Math. Soc., Providence, RI (2013).}



\bibitem[deJS06]{dJ-S:2fanos_1} {A. J. de Jong, J. Starr}, {{\em A note on Fano manifolds whose second Chern character is positive}},  {preprint arXiv:math/0602644v1 }, {(2006)}.


\bibitem[dJS07]{dJ-S:2fanos_2}
{A. J. de Jong,  J. Starr,  {\it Higher {F}ano manifolds and rational surfaces}, 
Duke Math. J, vol.139, no.1,  173--183 (2007).}

\bibitem[dJHS08]{dJ-S:homogeneous}
{A. J. de Jong, X. He e  J. Starr,  {\it Families of rationally simply connected varieties over surfaces
and torsors for semisimple groups}, 
Publications Math\'ematiques de l'IH\'ES, Tome 114 (2011) , pp. 1--85.}

\bibitem[GPPS22]{Perrin+}
{R. Gonzales, C. Pech, N. Perrin and A. Samokhin},
{\it Geometry of horospherical varieties of {P}icard rank one},
{Int. Math. Res. Not. IMRN},
{12},
{8916--9012}
{(2022)}.

\bibitem[GHS03]{GHS}
{T. Graber, J. Harris, J. Starr, J,
{\it Families of rationally connected varieties}, 
J. Amer. Math. Soc., vol.16, no.1, 57--67 (2003).}

\bibitem[Kan21]{Kanemitsu}
{A. Kanemitsu, {\it Fano manifolds and stability of tangent bundles}, 
Journal f\"ur die reine und angewandte Mathematik, vol. 2021, no. 774, 163--183 (2021).}

\bibitem[KMM92]{KMM}
{J. Koll{\'a}r, Y. Miyaoka and S. Mori},
{{\it Rational connectedness and boundedness of {F}ano manifolds}},
{J. Differential Geom.},
{Vol. \bf 36},
{(1992)},
{3},
{765--779.}

\bibitem[Nob11]{Nobili2011}
{E. E. Nobili},
{\it Classification of {T}oric 2-{F}ano 4-folds},
{Bulletin of the Brazilian Mathematical Society, New Series},
{Vol. \bf 42},
{(2011)},
{3},
{399--414.}




\bibitem[Pas09]{Pasquier}
{B. Pasquier, {\it  On some smooth projective two-orbit varieties with {P}icard number 1}, Math.
Ann., 344(4), 963--987 (2009).}

\bibitem[PW95]{PW}
{T. Peternell and J. Wi\'{s}niewski},
{\it On stability of tangent bundles of {F}ano manifolds with $b_2=1$},
{J. Algebraic Geom.},
{Vol. \bf 4},
{(1995)},
{no. 2},
{363--384}.

\bibitem[Sat12]{Sato2012}
{H. Sato},
{\it The numerical class of a surface on a toric manifold},
{Int. J. Math. Math. Sci.},
{Art. ID 536475, 9},
{(2012)}.


\bibitem[Sat16]{Sato2016},
{H. Sato},
{\it Toric 2-{F}ano manifolds and extremal contractions},
{Proc. Japan Acad. Ser. A Math. Sci.},
{Vol. \bf 92},
{No. 10},
{121--124}
{(2016)}.

\bibitem[SS20]{SatoSuyama2020},
{H. Sato and Y. Suyama},
{\it  Remarks on toric manifolds whose {C}hern characters are positive},
{Comm. Algebra},
{Vol. \bf 48},
{No. 6},
{2528--2538}
{(2020).}

\bibitem[SSS21]{SanoSatoSuyama2020},
{Y. Sano, H. Sato and Y. Suyama},
{\it  Toric {F}ano manifolds of dimension at most eight with positive second {C}hern characters},
{Kumamoto J. Math.},
{Vol. \bf 34},
{1--13}
{(2021).}




		


\end{thebibliography}
\end{document}